\documentclass[10pt, letterpaper, conference]{IEEEconf}
\IEEEoverridecommandlockouts
\usepackage{cite}
\usepackage{amsmath,amssymb,amsfonts}
\usepackage{algorithmic}
\usepackage{graphicx}
\usepackage{textcomp}
\usepackage{xcolor}

\newtheorem{definition}{\upshape\bfseries Definition}

\newtheorem{corollary}{\upshape\bfseries Corollary}
\newtheorem{lemma}{\upshape\bfseries Lemma}
\newtheorem{proposition}{\upshape\bfseries Proposition}
\newtheorem{theorem}{\upshape\bfseries Theorem}

\usepackage{bm}

\newcommand{\Natn}{\mathbb{N}}

\DeclareMathOperator{\rank}{rank}
\DeclareMathOperator{\img}{im}              
\renewcommand{\subset}{\subseteq}
\newcommand{\range}[1]{[\![#1]\!]}          
\newcommand{\infconv}[2]{}
\newcommand{\response}{\bm{\gamma}}

\newcommand{\OO}{\bm{\mathcal{O}}}       
\newcommand{\ctrb}{\bm{\mathcal{C}}}       
\renewcommand{\AA}{\bm{A}}          
\newcommand{\BB}{\bm{B}}            
\newcommand{\CC}{\bm{C}}            
\newcommand{\DD}{\bm{D}}            
\newcommand{\HH}{\bm{H}}            
\newcommand{\PPsi}{\bm{\Psi}}       
\newcommand{\GGamma}{\bm{\Gamma}}   
\newcommand{\states}{\mathbb{R}^n}
\newcommand{\weakuo}{\mathcal{V}}

\newcommand{\inputs}{\langle \Delta_s \rangle}
\newcommand{\dinputs}{\langle \Delta_{2s} \rangle}
\newcommand{\spanof}[1]{\langle #1 \rangle}

\newcommand{\xx}{\bm{x}}            
\newcommand{\uu}{\bm{u}}            
\newcommand{\vv}{\bm{v}}            

\makeatletter
\let\NAT@parse\undefined
\makeatother

\usepackage{etoolbox}
\makeatletter
\patchcmd{\@begintheorem}{\textit}{\textbf}{}{}
\makeatother

\usepackage{hyperref}
\hypersetup{
    colorlinks=true,
    linkcolor={red!50!black},
    citecolor={blue!80!black},
    urlcolor={blue!80!black}
}

\begin{document}

\title{\LARGE \bf Invertibility of Discrete-Time Linear Systems with Sparse Inputs
}

\author{Kyle Poe, Enrique Mallada, and Ren\'e Vidal 
\thanks{The authors thank the support of the ONR MURI grant 503405-78051, Control and Learning Enabled Verifiable Robust AI (CLEVR-AI)}
\thanks{K. Poe is with the Applied Mathematics and Computational Science Group, University of Pennsylvania,
PA 19104
        {\tt\small kpoe@sas.upenn.edu}}
\thanks{E. Mallada is with the Department of Electrical and Computer Engineering, Johns Hopkins University,
MD 21218
        {\tt\small mallada@jhu.edu}}
\thanks{R. Vidal is with the Departments of Electrical and Systems Engineering, Radiology, Computer and Information Science, Statistics and Data Science, University of Pennsylvania,
PA 19104
        {\tt\small vidalr@seas.upenn.edu}}
}

\maketitle

\begin{abstract}
One of the fundamental problems of interest for discrete-time linear systems is whether its input sequence may be recovered given its output sequence, a.k.a. the left inversion problem. Many conditions on the state space geometry, dynamics, and spectral structure of a system have been used to characterize the well-posedness of this problem, without assumptions on the inputs. However, certain structural assumptions, such as input sparsity, have been shown to translate to practical gains in the performance of inversion algorithms, surpassing classical guarantees. Establishing necessary and sufficient conditions for left invertibility of systems with sparse inputs is therefore a crucial step toward understanding the performance limits of system inversion under structured input assumptions. In this work, we provide the first necessary and sufficient characterizations of left invertibility for linear systems with sparse inputs,  echoing classic characterizations for standard linear systems. The key insight in deriving these results is in establishing the existence of two novel geometric invariants unique to the sparse-input setting, the weakly unobservable and strongly reachable subspace arrangements. By means of a concrete example, we demonstrate the utility of these characterizations. We conclude by discussing extensions and applications of this framework to several related problems in sparse control.
\end{abstract}

\section{Introduction}
Dynamical systems, ubiquitous in modern applied mathematics, are often characterized by their intrinsic properties, such as stability and orbit structure. However, in many instances, one is less interested in the system itself, and more interested in the relationship the system induces between some \emph{input} time series, and the resulting \emph{observations}, usually produced as some state-dependent function of the inputs. Nowhere is the importance of this fundamental relationship more apparent than in signal processing, where one observes some transformed or corrupted version of an input signal, and wishes to recover the original signal. 

In systems theory, the ability to uniquely recover a sequence of inputs given a sequence of outputs, or demonstrate the existence of a sequence of inputs which produces a desired output, is called \emph{invertibility} \cite{sain1969invertibility, hirschorn1979invertibility}. In addition to guaranteeing the well-posedness of signal recovery problems, left invertibility is a necessary condition for the existence of unknown input observers, and is closely related to fault detection \cite{liu1994detection, edelmayer2004input} and input observability \cite{hou1998input, martinelli2019nonlinear}. 

In the linear setting, several equivalent and interpretable characterizations of invertibility providing different perspectives on the problem have been established. In particular, the geometric approach emphasizes certain invariant subspaces of the system which give rise to the so-called special coordinate basis \cite{sannuti1987special}, which has been used to great practical effect in observer design \cite{tranninger2023unknown}. Another characterization focuses on the algebraic properties of the input-output map, establishing a clear path to the design of delayed system inverses \cite{sain1969invertibility, ansari2019deadbeat}. Yet another characterization concerns the zeros of the Rosenbrock system matrix \cite{trentelman2002control}, providing a connection to transfer function methods and the stronger problem of stable invertibility \cite{moylan1977stable}. These differing characterizations illuminate complementary aspects of the system, and have been used to great effect in the analysis of various recovery problems.

In recent years, there has been an increased awareness of the benefit of modeling input structure at the heart of various control applications. Notably, sparsity in the inputs has been explored in some depth \cite{sefati2015linear, poe2023necessary, chakraborty2023joint, sriram2020control}, particularly in the context of networked systems \cite{florez2020structural, zhang2020state}. Various algorithms have been developed for online recovery of sparse signals \cite{charles2016dynamic, charles2020efficient, joseph2015online}, and more recently, for the recovery of sparse inputs to general linear systems \cite{sefati2015linear, joseph2017noniterative, chakraborty2023joint}. Necessary and sufficient conditions for the finite-horizon case have been established in \cite{poe2023necessary}, a problem for which a performant Bayesian recovery algorithm was introduced in \cite{chakraborty2023joint}.

Despite the empirical success of these approaches, characterizations of the well-posedness of the infinite-horizon left invertibility problem for linear systems with sparse inputs have yet to emerge. Such characterizations are vital in particular for applications to structured and networked systems, in which the minimial required delay for input recovery may be generically nontrivial \cite{garin2022generic}.

In this work, we establish necessary and sufficient characterizations of left invertibility for linear systems with sparse inputs, which parallel established characterizations for standard linear systems. Our contributions are as follows: 
\begin{enumerate}
    \item We introduce the notion of weakly unobservable and strongly reachable subspace arrangements, generalizing key invariants of classical geometric linear systems theory to the sparse input setting, and show that these objects can be computed and used to directly certify invertibility.
    \item We establish rank-based conditions for left  sparse invertibility, and show that if an inverse exists, it may be realized with finite delay.
    \item The invertibility of systems with sparse inputs having temporally periodic support patterns is characterized via the zeros of a generalized Rosenbrock matrix, and this construction is used to provide a final necessary and sufficient condition for left invertibility under a generic sparse input assumption.
    \item We present an example to illustrate application of these ideas, and conclude by discussing extensions and connections to related problems in sparse control.
\end{enumerate}

\section{Preliminaries}
In this section, we first introduce the basic notation that will be used throughout the paper (see Section \ref{sec:notation}). Then in Section \ref{sec:linear-systems}, we overview several necessary and sufficient conditions for left invertibility of linear systems, including geometric, rank-based and spectral characterizations. Finally, in Section \ref{sec:subspace-arrangements} we review basic properties of subspace arrangements that are needed to extend classical invertibility results to linear systems with sparse inputs.

\subsection{Notation}
\label{sec:notation}
We denote $\Natn := \{0, 1, 2, \ldots\}$, and for any natural number $N \in \Natn$, $\range{N} := \{0, 1, \ldots, N-1\}$.
Let $\Omega, \Omega'$ be sets. We define the product $\Omega \times \Omega' = \{(\omega, \omega'): \omega \in \Omega, \omega' \in \Omega'\}$, and identify $\Omega^N$ with the set of functions $\{\omega: \range{N} \to \Omega\}$; and equivalently, the set of $N$-tuples $(\omega_0, \omega_1, \ldots, \omega_{N-1})$. $|\Omega|$ is the cardinality of $\Omega$.  For a sequence $\omega: \Natn \to \Omega$, we denote by $[\omega]_k^N \in \Omega^N$ the tuple $(\omega_k, \omega_{k+1}, \ldots, \omega_{k + N -1})$, which may be read as the \emph{first $N$ elements of $\omega$ beginning at $k$}. We also define the shift operator $\sigma$ on tuples and sequences such that $\forall k, \sigma(\omega)_k = \omega_{k+1}$; note that $\sigma: \Omega^N \to \Omega^{N-1}$  for tuples and $\sigma: \Omega^\Natn \to \Omega^{\Natn}$ for sequences. We liberally define $0$ to be the zero element in the relevant context. 
We denote the $k$th canonical Euclidean basis vector by $\bm{e}_k \in \mathbb{R}^n$, the canonical subspace associated with $S \subset \range{m}$ as $\spanof{S} := \text{span}\{\bm{e}_i\}_{i \in S}$, and the preimage of a set $\mathcal{V}$ under a linear map/matrix $\AA$ as $\AA^{-1}\mathcal{V} := \{\xx: \AA\xx \in \mathcal{V}\}$.
If $\Delta$ is a set of sets, given $S, T: \Natn \to \Delta$, we define the sequence $S \cup T := (S_k \cup T_k)_{k \in \Natn}$. If $\Delta$ consists of subsets of the support of $\AA$, for $S \in \Delta$, $\AA_S$ is the matrix consisting of columns indexed by $S$.
For a block matrix $\bm{\Theta}_N$ with $N$ block columns $\{\bm{\Theta}_i\}_{i \in \range{N}}$ and $S \in \Delta^N$, we define the matrix $\bm{\Theta}_S$ to be the block matrix with $N$ block columns $\{(\bm{\Theta}_i)_{S_i}\}_{i \in \range{N}}$. 

\subsection{Linear Systems}
\label{sec:linear-systems}
Throughout, we consider a fixed, finite-dimensional, discrete time linear system $\Sigma = (\AA, \BB, \CC, \DD) \in \mathbb{R}^{n\times n} \times \mathbb{R}^{n \times m} \times \mathbb{R}^{p \times n} \times \mathbb{R}^{p \times m}$. To this system, for any $N > 0$, we associate the following block matrices; respectively, the \emph{observability}, \emph{finite response}, and \emph{controllability} matrices:
    \begin{align*}
        \OO_N \!&=\!\! \begin{bmatrix}
            \CC \\ \CC\AA \\ \CC\AA^2\\ \vdots \\ \CC\AA^{N-1}
        \end{bmatrix}\!, 
        \GGamma_N \!=\!\! \begin{bmatrix}
            \DD \\
            \CC\BB \!&\! \DD \\
            \CC\AA\BB \!&\! \CC\BB \!&\! \DD \\
            \vdots \!&\!&\! \ddots\!&\! \ddots \\
            \CC\AA^{N-2}\BB \!&\! \cdots \!&\!&\! \CC\BB\!&\! \DD
        \end{bmatrix} \! .\\
        \ctrb_N \!&=\! \begin{bmatrix}
            \AA^{N-1}\BB & \cdots & \AA^2\BB & \AA\BB & \BB
        \end{bmatrix}
        .
    \end{align*}
    Note that our definition of the controllability matrix has a reversed order of powers from the typical definition; this permits concise identities like the following, for any $M < N$:
    \begin{equation}        
        \GGamma_{N} = \begin{bmatrix}
            \GGamma_{N-M} & 0 \\ \OO_M\ctrb_{N-M} & \GGamma_{M}
        \end{bmatrix} = \begin{bmatrix}
            \DD & 0 \\ \OO_{N-1}\BB & \GGamma_{N-1}
        \end{bmatrix} .
    \end{equation}

    Given an input sequence $\uu: \Natn \to \mathbb{R}^m$, we define the \emph{infinite response} of $\Sigma$ from $\xx_0 \in \mathbb{R}^n$ to be the sequence $\response(\xx_0, \uu): \Natn \to \mathbb{R}^p$ such that for any $N, k$, the following conditions hold:
    \begin{align}
        [\response(\xx_0, \uu)]_k^N 
        &= \OO_N\xx_k + \GGamma_N[\uu]_k^N,
    \\
        \sigma^k(\response(\xx_0, \uu)) 
        &= \response(\AA^k\xx_0 + \ctrb_k[\uu]_0^k, \sigma^k(\uu)).
    \end{align}
    Note that $(\response(\xx_0, \uu)_k)_{k \in \Natn}$ is just the sequence of outputs for the linear system with initial state $\xx_0$ and inputs $(\uu_k)_{k \in \Natn}$, and that $\response$ is linear:  $\response(\alpha(\xx_0, \uu) + \beta(\xx_0', \uu')) = \alpha \response(\xx_0,\uu) + \beta \response(\xx_0', \uu')$. To denote the zero-state response, we write $\response_0(\uu) := \response(0, \uu)$.

The geometric approach to analyzing system invertibility emphasizes certain subspaces linking output behavior to choice of inputs:
    \begin{definition}[Invariant Subspaces]\,
        \begin{itemize}
            \item The \emph{weakly unobservable subspace} $\mathcal{V}(m) \subset \mathbb{R}^n$ consists of $\xx$ s.t. $ \exists \uu: \Natn \to \mathbb{R}^m$, $\response(\xx, \uu) = 0$.
            \item The \emph{strongly reachable subspace} $\mathcal{T}(m) \subset \mathbb{R}^n$ consists of $\xx$ s.t. $\exists \uu: \Natn \to \mathbb{R}^m, \exists N \in \Natn$, $\GGamma_N[\uu]_0^N = 0$ and $\ctrb_N[\uu]_0^N = \xx$. 
        \end{itemize}
    \end{definition}

Note that $\mathcal{V}(m)$ contains the unobservable subspace $\ker \OO_n$, and $\mathcal{T}(m)$ is a subspace of the reachable subspace $\img \ctrb_n$. They are thus naturally viewed as respectively weakened and strengthened versions of these spaces, accounting for particular choices of inputs.

An alternative viewpoint of the input-output characteristics of a linear system involves its invariant zeros, which are defined as the $z \in \mathbb{C}$ such that the following matrix pencil drops rank:

\begin{definition}[Rosenbrock System Matrix]
    \begin{equation}        
        \bm{R}(z) = \begin{bmatrix}
            \AA - z\bm{I} & \BB \\ \CC & \DD
        \end{bmatrix}
    \end{equation}
\end{definition}
    \vspace{1em}
    We now recall the definition of left invertibility for linear systems \cite{sain1969invertibility}:
    \begin{definition}[Left Invertibility]
        We say that $\Sigma$ is \emph{left invertible} if for any $\uu, \uu': \Natn \to \mathbb{R}^m$, $\response_0(\uu) = \response_0(\uu')$ implies $\uu = \uu'$.
    \end{definition}
    
    There are several equivalent characterizations of when left invertibility holds, which can be considered as arising from complementary \emph{perspectives} on what it means for a system to be invertible. In the case of left invertibility:

    \begin{theorem}[Left Invertibility, \cite{trentelman2002control, sain1969invertibility}]\label{thm:leftinvlinear}
    The following are equivalent:
        \begin{enumerate}
            \item $\Sigma$ is left invertible.
            \item $\weakuo(m) \cap \mathcal{T}(m) = 0$ and $\ker \begin{bmatrix}
                \BB \\ \DD
            \end{bmatrix} = 0$.
            \item $\exists N$, $\rank \GGamma_N - \rank \GGamma_{N-1} = m$.
            \item $\exists z \in \mathbb{C}$, $\rank \bm{R}(z) = n + m$.
        \end{enumerate}
    \end{theorem}

One may think of conditions (2-4) as providing \emph{geometric, rank based}, and \emph{spectral} characterizations of left invertibility: (2) indicates that no input is invisible to measurement results in a weakly unobservable perturbation to the state; (3) indicates that if an inverse exists, it can be implemented with some finite delay; and (4) can be used to show that the associated transfer function matrix admits a rational polynomial left inverse.

\subsection{Subspace Arrangements}
\label{sec:subspace-arrangements}
In the sparse recovery literature, one often considers the set of all vectors with a particular sparsity pattern. If we denote $\Delta_s := \{S \subset \range{m}: |S| \le s\}$, this set of vectors may be written as $\bigcup_{S \in \Delta_s} \spanof{S}$. Generally speaking, for any set $\Delta$ of subsets of $\range{m}$, we will write
\begin{equation}
    \spanof{\Delta} := \bigcup_{S \in \Delta} \spanof{S} \subset \mathbb{R}^m
\end{equation}
This object is an example of a \emph{finite subspace arrangement}:

\begin{definition}[Finite Subspace Arrangement]
    Let $\mathcal{A} \subset \mathbb{R}^n$. We call $\mathcal{A}$ a \emph{finite subspace arrangement in $\mathbb{R}^n$} if there exists an natural number $c$ and a collection of subspaces $(U_i)_{i \in \range{c}}, U_i \subset \mathbb{R}^n$, such that $\mathcal{A} = \bigcup_{i \in \range{c}} U_i$. We denote the smallest such $c$ as $c(\mathcal{A})$, and call this the \emph{size} of $\mathcal{A}$.
\end{definition}

For example, the set of $s$-sparse vectors $\spanof{\Delta_s}$ in $\mathbb{R}^m$ has size $\binom{m}{s}$. While this case only contains subspaces of dimension $s$, in general, subspace arrangements may contain subspaces of differing dimension.
\begin{definition}[Dimension Vector]
    Let $\mathcal{A} \subset \mathbb{R}^n$ be a finite subspace arrangement of size $c$ such that $\mathcal{A} = \bigcup_{i \in \range{c}} U_i$. Define $d_k(\mathcal{A}) = |\{i \in \range{c}: \dim U_i = k\}|$, and $\bm{d}(\mathcal{A}) = (d_1(\mathcal{A}), \ldots, d_n(\mathcal{A})) \in \Natn^n$. We refer to $\bm{d}(\mathcal{A})$ as the \emph{dimension vector} of $\mathcal{A}$.
\end{definition}

Note that for a linear subspace $\mathcal{X} \subset \mathbb{R}^n$, we have that $\bm{d}(\mathcal{X}) = \bm{e}_{\dim(\mathcal{X})} \in \mathbb{N}^n$. This provides a means of assessing the relative size of two subspace arrangements, in a similar fashion as simple dimension for subspaces:

\begin{definition}[Dimensional Order]
    Define $(\Natn^n, \preceq)$ the totally ordered set such that for $\bm{a}, \bm{b} \in \Natn^n$, $\bm{a} \preceq \bm{b}$ if for $j = \max\{i \in \range{n}: \bm{a}_i \ne \bm{b}_i\}, \bm{a}_j \le \bm{b}_j$, and $\bm{a} \prec \bm{b}$ analogously. 
    Given subspace arrangements $\mathcal{A}, \mathcal{B} \subset \mathbb{R}^n$, we define the \emph{dimensional order} $\preceq$ such that $\mathcal{A} \preceq \mathcal{B}$ if  $\bm{d}(\mathcal{A}) \preceq \bm{d}(\mathcal{B})$.
\end{definition}

Lastly, we note that if $\mathcal{U}, \mathcal{V} \subset \mathbb{R}^n$ are finite subspace arrangements and $\HH$ is a well-defined linear map, then $\mathcal{U} \times \mathcal{V}, \mathcal{U} \cup \mathcal{V}, \mathcal{U} \cap \mathcal{V}, \mathcal{U} + \mathcal{V}, \HH \mathcal{U}, \HH^{-1}\mathcal{U}$ are all finite subspace arrangements.

\section{Invertibility of Linear Systems with Sparse Inputs}\label{section:inv}
In this section, we will generalize the classical notion of left invertibility to systems with piecewise $s$-sparse inputs:
\begin{definition}[Left $\mathcal{U}$-Invertibility]\label{def:leftsp}
    Let $\mathcal{U}\subset (\mathbb{R}^m)^\Natn$.
     We say that $\Sigma$ is left $\mathcal{U}$-invertible if $\forall \uu, \uu' \in \mathcal{U}, \quad \response_0(\uu) = \response_0(\uu') \implies \uu = \uu'$. We will say that $\Sigma$ is \emph{left $s$-sparse invertible} when $\mathcal{U} = \{\uu: \Natn \to \spanof{\Delta_s}\}$.
\end{definition}

One can interpret this as a statement about the \emph{injectivity} of $\response_0$ when restricted to a given input class $\mathcal{U}$. The remainder of this section is dedicated to establishing analogous conditions for $s$-sparse invertibility to those in Theorem \ref{thm:leftinvlinear}, beginning with the characterization of sparse counterparts to the weakly unobservable and strongly reachable subspaces.

\subsection{Geometric Characterization}
The notion of weak unobservability generalizes immediately to the sparse setting:
\begin{definition}[Weakly Unobservable Point]
    Let $\xx \in \states$. If there exists $\uu: \Natn \to \inputs $ such that $\response(\xx, \uu) = 0$, then we call $\xx$ \emph{weakly $s$-sparse unobservable}, and denote by $\weakuo(s)$ the set of all such $\xx$.
\end{definition}

We will proceed to show that $\weakuo(s)$ is a finite subspace arrangement.
To compute $\weakuo(s)$, consider the following set mapping:
\begin{equation}\label{alg1}
    f_s(\mathcal{A}) = \mathcal{A} \cap \begin{bmatrix}
        \CC \\ \AA
    \end{bmatrix}^{-1} \left(0 \times \mathcal{A} + \begin{bmatrix}
        \DD \\ \BB
    \end{bmatrix}\inputs\right)
\end{equation}

Intuitively, $f_s(\mathcal{A})$ returns the set of $\xx \in \mathcal{A}$ such that there exists an $s$-sparse input $\uu$ satisfying $\AA\xx + \BB\uu \in \mathcal{A}$ and $\CC\xx + \DD\uu = 0$. Note that, by construction, for any $\mathcal{A} \subset \mathbb{R}^n$, $f_s(\mathcal{A}) \subset \mathcal{A}$; and if $f_s(\mathcal{A}) = \mathcal{A}$, $f_s^2(\mathcal{A}) := (f_s \circ f_s)(\mathcal{A}) = \mathcal{A}$.

\begin{lemma}
    $\forall k > 0$, $f_s^k(\mathbb{R}^n) = \OO_k^{-1}(\GGamma_k \inputs^k)$
\end{lemma}
\begin{proof}
We first remark $\OO_{k+1}^{-1}(\GGamma_{k+1}\spanof{\Delta_s}^{k+1}) \subset \OO_k^{-1}(\GGamma_k\spanof{\Delta_s}^k)$, this is readily seen by considering that $\OO_{k+1}\xx \in \GGamma_{k+1}\spanof{\Delta_s}^{k+1} \implies \OO_{k}\xx \in \GGamma_{k}\spanof{\Delta_s}^{k}$.
    The remaining proof is by induction on $k$. For $k=1$, $f_s(\mathbb{R}^n) = \CC^{-1}\DD\inputs = \OO_1^{-1}\GGamma_1\inputs$. Now suppose $f_s^k(\mathbb{R}^n) = \OO_k^{-1}(\GGamma_k\inputs^k)$:
    \begin{align*}
        &\OO_{k+1}^{-1}(\GGamma_{k+1}\inputs^{k+1})\\
        &=\begin{bmatrix}
            \CC \\
            \OO_{k}\AA
        \end{bmatrix}^{-1}\left(
        \begin{bmatrix}
            \DD \\ \OO_k\BB
        \end{bmatrix}\inputs + \begin{bmatrix}
            0 \\ \GGamma_k
        \end{bmatrix}\inputs^k\right) \\
        &\overset{*}{=}\begin{bmatrix}
            \CC \\
            \AA
        \end{bmatrix}^{-1}\left(\begin{bmatrix}
            \bm{I} & 0 \\ 0 & \OO_k
        \end{bmatrix}^{-1}
        \begin{bmatrix}
            \DD \\ \OO_k\BB
        \end{bmatrix}\inputs\right. \\ &+ \left.\begin{bmatrix}
            \bm{I} & 0 \\ 0 & \OO_k
        \end{bmatrix}^{-1}\begin{bmatrix}
            0 \\ \GGamma_k
        \end{bmatrix}\inputs^k\right) \\
        &=\begin{bmatrix}
            \CC \\
            \AA
        \end{bmatrix}^{-1}\left(
        \begin{bmatrix}
            \DD \\ \BB
        \end{bmatrix}\inputs + 0 \times \OO_k^{-1}(\GGamma_k\inputs^k)\right) \\
        &= f_s(f_s^k(\mathbb{R}^n)) = f_s^{k+1}(\mathbb{R}^n)
    \end{align*}
    Where $*$ follows from the subspace identity $\AA^{-1}(\mathcal{B} + \mathcal{C}) = \AA^{-1}\mathcal{B} + \AA^{-1}\mathcal{C}$ when $\mathcal{B} \subset \img \AA$ and distributivity of preimage over unions, and the last step follows from the inductive hypothesis together with the fact that $\OO_{k+1}^{-1}(\GGamma_{k+1}\inputs^{k+1}) \subset \OO_{k}^{-1}(\GGamma_{k}\inputs^{k})$.
\end{proof}

\begin{proposition}\label{prop:finite}
    For every $s$, $\mathcal{V}(s)$ is a finite subspace arrangement. Furthermore, defining
    \begin{equation}
        \mathcal{V}_k(s) := \OO_k^{-1}(\GGamma_k \langle \Delta_s\rangle^k),
    \end{equation}
    there exists $N$ such that $\mathcal{V}_N(s) = \mathcal{V}_{N+1}(s) = \mathcal{V}(s)$. We define the \emph{weak $s$-sparse observability index} $\nu_s$ to be the smallest such $N$.
\end{proposition}

\begin{proof}
    Suppose that there exists $N$ such that $f_s^N(\mathbb{R}^n) = f_s^{N+1}(\mathbb{R}^n) = f_s(f_s^N(\mathbb{R}^n))$, then $\mathcal{V}^* := f_s^N(\mathbb{R}^n)$ is a fixed point of $f_s$. Hence, suppose $\xx \in \mathcal{V}^*$, it follows that there exists $\uu_0 \in \langle \Delta_s \rangle$ such that $\AA\xx + \BB\uu_0 \in \mathcal{V}^*$ and $\CC\xx + \DD\uu_0 = 0$, hence we we may construct $\uu: \mathbb{N} \to \langle \Delta_s \rangle$ such that $\gamma_{\xx}(\uu) = 0$, so $\xx \in \mathcal{V}(s)$.

    Fix $s$, and denote $\weakuo_k := f_s^k(\states) = \mathcal{V}_k(s)$.
    Consider that, if $\mathcal{A}, \mathcal{B}$ are subspace arrangments, $\mathcal{A} \subsetneq \mathcal{B} \implies d(\mathcal{A}) \prec d(\mathcal{B})$.
    Note that 
    $c(\weakuo_{k+1}) \le c(\weakuo_{k})C(m, s)$, as there are at most $c(\langle \Delta_s \rangle)$ subspaces in $\weakuo_{k+1}$ for every subspace in $\weakuo_k$. 

    Suppose that for some $k$, there exists $\mathcal{A}$ such that $\mathcal{V}_k \subset \mathcal{A}$ and $\mathcal{V}_k \ne \mathcal{V}_{k+1}$. Suppose that $f_s(\mathcal{A}) = \mathcal{A}$. We have that $\mathcal{A} \subset \mathbb{R}^n$, so $\mathcal{A} = f_s^k(\mathcal{A}) \subset \mathcal{V}_k \implies \mathcal{A} = \mathcal{V}_k$. But $\mathcal{V}_k \ne \mathcal{V}_{k+1}$, this is a contradiction, hence if $\mathcal{V}_k \ne \mathcal{V}_{k+1}$, $\weakuo_k \subset \mathcal{A} \implies f_s(\mathcal{A}) \subsetneq \mathcal{A}$, and therefore $d(f_{s}(\mathcal{A})) \prec d(\mathcal{A})$.

    By way of contradiction, suppose there does not exist $N$ such that $\weakuo_N = \weakuo_{N+1}$. Then the sequence $\weakuo_k$ is strictly decreasing by inclusion in $k$, and therefore strictly decreasing in dimension. By the above, $\weakuo_k \subset \mathcal{A} \implies d(f_s(\mathcal{A})) \prec d(\mathcal{A})$. We show a contradiction by induction on $\mu(\mathcal{V}_k) = \min\{i: d_i(\weakuo_k) \ne 0\}$.

    Suppose $\exists k, (d_1, \ldots, d_n)$ such that $d_1 \ne 0$ and $d(\weakuo_k) \preceq (d_1, \ldots, d_n)$. Since $f_s(\weakuo_k) \prec \weakuo_k$, $d(f_s^{d_1}(\weakuo_k)) = d(\weakuo_{k+d_1}) \preceq (0, d_2, \ldots, d_n)$.

    Now assume, by induction, that if $\exists k, \bm{d}$ such that $\bm{d} = (0, \ldots, 0, d_{r-1}, d_r, \ldots, d_n)$ and $d(\weakuo_k) \le \bm{d}$, that $\exists N$ such that $d(\weakuo_{k+N}) \preceq (0, \ldots, 0, d_r, d_{r+1}, \ldots, d_n)$. Suppose that $\exists k, \bm{d}$ such that $\bm{d} = (0, \ldots, 0, d_r, \ldots, d_n)$ and $d(\weakuo_k) \preceq \bm{d}$. Let $\mathcal{A}$ be such that $\weakuo_k \subset \mathcal{A}$, and $d(\mathcal{A}) = \bm{d}$. Then since $c(f(\mathcal{A})) \le c(\mathcal{A})c(\inputs)$, $d(\weakuo_{k+1}) \preceq d(f(\mathcal{A})) \preceq (0, \ldots, c(\mathcal{A})(c(\inputs)-  1) + 1, d_r^k-1, \ldots, d_n^k) = \bm{d}'$. Then there exists $k' = k+1, \bm{d}'$ such that $d(\weakuo_{k'}) \preceq \bm{d}'$ and $\bm{d}' = (0, \ldots, 0, d_{r-1}', d_r-1, d_{r+1}, \ldots, d_n)$. So, by the inductive hypothesis, there exists $N$ such that $d(\weakuo_{k' + N}) \preceq (0, \ldots, 0, d_r-1, d_{r+1}, \ldots, d_n)$. By repeated application of this fact, $\exists N'$ such that $d(\weakuo_{k' + N'}) \preceq (0, \ldots, 0, 0, d_{r+1}, \ldots, d_n)$.

    It follows by induction that as $d(\weakuo_0) = d(\mathbb{R}^n) = (0, \ldots, 1)$, there exists $N$ such that $d(\weakuo_N) \preceq 0$. But then $\weakuo_N = 0$, and so $\weakuo_{N+1} = \weakuo_N$, a contradiction.

    To see that the resulting subspace arrangement is finite, it suffices to note that the subspace arrangement is of size at most $\binom{m}{s}^{\nu_s}$.
\end{proof}

The set of strongly $s$-sparse reachable points likewise is readily generalized from the linear case:

\begin{definition}
    Let $\xx \in \mathbb{R}^n$. If there exists $\uu: \Natn \to \inputs$ and $N$ such that $\xx = \ctrb_N[\uu]_0^N$ and $\GGamma_N[\uu]_0^N = 0$, then $\xx$ is said to be \emph{strongly $s$-sparse reachable}. We denote the set of all such $\xx$ as $\mathcal{T}(s)$.
\end{definition}

We omit the proof for the following, as it follows from essentially the same argument as for $\mathcal{V}(s)$:

\begin{proposition}
    For every $s$, $\mathcal{T}(s)$ is a finite subspace arrangement. Furthermore, defining 
    \begin{equation}
        \mathcal{T}_k(s) := \ctrb_k(\ker \GGamma_k \cap \langle \Delta_s\rangle^k)
    \end{equation}
    there exists $N$ such that $\mathcal{T}_N(s) = \mathcal{T}_{N+1}(s) = \mathcal{T}(s)$. 
    We define the \emph{strong $s$-sparse reachability index} $\tau_s$ to be the smallest such $N$.
\end{proposition}

It may be in turn shown that $\mathcal{T}(s)$ is obtained as the fixed point of an iterated set map, as in \eqref{alg1}

\begin{corollary}
    $\mathcal{T}_k(s)$ satisfies the recursion
    \begin{equation}\label{alg2}
        \mathcal{T}_{k+1}(s) = \begin{bmatrix}
            \AA & \BB
        \end{bmatrix}\left((\mathcal{T}_k(s) \times \spanof{\Delta_s})\cap \ker \begin{bmatrix}
            \CC & \DD
        \end{bmatrix}\right)
    \end{equation}
\end{corollary}
\vspace{1em}
Using these subspace arrangements, we may geometrically characterize left invertibility.

\begin{proposition}\label{propinv}
    The following are equivalent:
    \begin{enumerate}
        \item $\Sigma$ is left $s$-sparse invertible.
        \item $\mathcal{T}(2s) \cap \mathcal{V}(2s) = 0$ and $\forall S \in \Delta_{2s}, \ker \begin{bmatrix}
            \DD_S \\ \BB_S
        \end{bmatrix} = 0.$
        \item For any $S, T \in \Delta_{s}$, $\ker \DD_{S \cup T} \cap \BB^{-1}\mathcal{V}(2s) = 0$.
    \end{enumerate}
\end{proposition}

\begin{proof}
    $(1 \Rightarrow 2)$ Suppose $\Sigma$ is left $s$-sparse invertible. Toward contradiction, take $S \in \Delta_{2s}$ such that $\ker \DD_S \cap \ker \BB_S \ne 0$, then there exists $T, T' \in \Delta_s$ such that $T \cup T' = S$ and $\uu: \Natn \to \spanof{T}, \vv: \Natn \to \spanof{T'}$ not equal such that $\forall k \in \Natn, \DD(\uu_k - \vv_k) = 0$ and $\BB(\uu_k - \vv_k) = 0$, so necessarily $\gamma_0(\uu - \vv) = 0 \implies \gamma_0(\uu) = \gamma_0(\vv)$, this is a contradiction. Now suppose instead that $\exists \xx \in \mathcal{T}(2s) \cap \mathcal{V}(2s)$, then there exists $\uu, \vv: \Natn \to \spanof{\Delta_s}$ and $N \in \Natn$ such that $\GGamma_N[\uu - \vv]_0^N = 0$ and $\ctrb_N[\uu - \vv]_0^N = \xx \in \mathcal{V}(2s)$, and $\response(\xx, \sigma^N(\uu - \vv)) = 0$.

    $(2 \Rightarrow 3)$ Suppose $\ker \DD_{S \cup T} \cap \BB^{-1}\mathcal{V}(2s) \ne 0$, then there exists $\uu_0 \in \spanof{S}, \vv_0 \in \spanof{T}$ such that $\DD(\uu_0 - \vv_0) = 0$ and $\BB(\uu_0 - \vv_0) \in \mathcal{V}(2s)$. But then $\BB(\uu_0 - \vv_0) \in \mathcal{T}_1(2s) \cap \mathcal{V}(2s) \subset  \mathcal{T}(2s) \cap \mathcal{V}(2s)$, contradicting $(2)$.

    $(3 \Rightarrow 1)$ Let $\uu, \vv: \Natn \to \spanof{\Delta_s}$, and suppose $\response_0(\uu) = \response_0(\vv)$. Then $\DD(\uu_0 - \vv_0) = 0$ and $\response(\BB\uu_0, \sigma(\uu)) = \response(\BB\vv_0, \sigma(\vv)) \iff \response(\BB(\uu_0 - \vv_0), \sigma(\uu - \vv)) = 0$, so $\uu_0 - \vv_0 \in \BB^{-1}\mathcal{V}(2s)$. So by $(3)$, $\uu_0 = \vv_0$. Hence we conclude $\Sigma$ is left $s$-sparse invertible.
\end{proof}

\subsection{Rank-Based Characterization}
While informative from a geometric perspective, it is not clear how one could approach the problem of actually building an inverse system from the geometric characterization. By considering the rank of $\GGamma$ when restricted to piecewise-$2s$-sparse supports, we show that if the system is $s$-sparse invertible, then it is possible to construct an inverse with a finite delay.

\begin{proposition}\label{rankprop}
    The system $\Sigma$ is $s$-sparse invertible if and only if there exists $N < \nu_{2s}$ s.t. for any $S = (S_0, S_1, \ldots, S_N) \in \Delta_{2s}^{N+1}$, 
    \begin{equation}\rank \GGamma_S - \rank\GGamma_{\sigma(S)} = |S_0|.
    \end{equation}
    In this event, we say $\Sigma$ is $s$-sparse invertible with delay $N$.
\end{proposition}
\begin{proof}
    Recall that for $S \in \Delta_{2s}^{N+1}$, $\sigma(S) = (S_1, \ldots, S_N) \in \Delta_{2s}^N$.
    
    $(\Rightarrow)$ Note that $\rank\GGamma_S = \rank\begin{bmatrix}
        \DD_{S_0} & 0 \\ \OO_{N-1}\BB_{S_0} & \GGamma_{\sigma(S)}
    \end{bmatrix}$ is equal to $|S_0| - \dim \begin{bmatrix}
        \DD_{S_0} \\\OO_N\BB_{S_0}
    \end{bmatrix}^{-1}(0 \times \img\GGamma_{\sigma(S)}) + \rank \GGamma_{\sigma(S)}$, so this characterization is equivalent to showing that $\begin{bmatrix}
        \DD_{S_0} \\ \OO_N\BB_{S_0}
    \end{bmatrix}^{-1}(0 \times \img \GGamma_{\sigma(S)}) = 0$. Suppose that for all $N \in \Natn$, there exists $S \in \Delta_{2s}^{N+1}$ such that $\begin{bmatrix}
        \DD_{S_0} \\ \OO_N\BB_{S_0}
    \end{bmatrix}^{\smash{-1}}(0 \times \img \GGamma_{\sigma(S)}) \ne 0$. Then there exists $\uu_0, \vv_0 \in \spanof{\Delta_s}$ such that $\DD(\uu_0 - \vv_0) = 0$, and $\BB(\uu_0 - \vv_0) \in \OO_{N}^{-1}(\GGamma_{N}\spanof{\Delta_{2s}}^{N})$. But then for $N \ge \nu_{2s}$, $\BB(\ker \DD \cap \dinputs) \cap \weakuo(2s) \ne 0$, contradicting invertibility. 

    $(\Leftarrow)$ Suppose that there exists $N$ such that for any $S \in \Delta_{2s}^{N+1}$, the rank condition holds. Then for any $\uu, \uu': \Natn \to \langle \Delta_s \rangle$, $\GGamma_{N+1}[\uu - \uu']_k^{N+1} = 0 \implies \uu_k = \uu'_k$. Suppose $\forall j < k, \uu_j = \uu'_j$. It follows that $[\gamma_0(\uu - \uu')]_k^{N+1} = 0 \implies \OO_{N+1}(\ctrb_k[\uu - \uu']_0^k) + \GGamma_{N+1}[\uu - \uu']_k^{N+1} = \GGamma_{N+1}[\uu - \uu']_k^{N+1} = 0 \implies \uu_k = \uu_k'$. It follows by  strong induction on $k$ that $\gamma_0(\uu) = \gamma_0(\uu') \implies \uu = \uu'$.
\end{proof}

In light of this result, we will define the \emph{inherent $s$-sparse delay} $d_s$ of the system as follows:
    \[d_s := \min\{N: \forall S \in \Delta_{2s}^{N+1}, \rank \GGamma_S - \rank\GGamma_{\sigma(S)} = |S_0|\}\]\,
By definition, if $d_s$ is finite, $d_s < \nu_{2s}$. As is the case with the inherent delay of linear systems with generic inputs, $d_s$ provides a lower bound on the delay of any $s$-sparse inversion algorithm. 

\subsection{Spectral Characterization}
The characterization of invertibility based on the Rosenbrock matrix is unique in its apparent simplicity, relying on no complicated block matrices or subspaces far removed from basic system parameters. Unfortunately, this simplicity prevents it from being able to capture the complex properties of changing input support patterns. To obtain a spectral characterization of left $s$-sparse invertibility, it is thus necessary to work with a version of the Rosenbrock matrix generalized to a pattern of $\tau$ supports $S \in \{T: T \subset \range{m}\}^\tau$:
\begin{equation}
    \bm{R}_S(z) := \begin{bmatrix}
        \AA^\tau - z\bm{I} & \ctrb_S \\ \OO_\tau & \GGamma_S
    \end{bmatrix}.
\end{equation}
Before addressing the general case, it is worth considering what the properties of this matrix can tell us about invertibility of the system over the set of $s$-piecewise sparse inputs with \emph{$\tau$-periodic} supports, that is:
\[
    \mathcal{U}_\tau(s) := \{\bm{u}: \Natn \to \spanof{\Delta_s}: \exists S \in \Delta_s^\tau, \uu_k \in \spanof{S_{k\text{ mod }\tau}}\}.
\]

\begin{lemma}
    If there exists $S \in \Delta_{2s}^\tau$ such that $\begin{bmatrix}
        \ctrb_S \\ \GGamma_S
    \end{bmatrix}$ is rank deficient, then $\Sigma$ is not left $\mathcal{U}_\tau(s)$-invertible.
\end{lemma}

\begin{proof}
    Suppose that for some $S$, $\begin{bmatrix}
        \ctrb_S \\ \GGamma_S
    \end{bmatrix}$ is not full rank. Then there exists $\uu, \uu' \in \mathcal{U}_\tau(s)$, $\forall k \ge \tau, \uu_k = \uu'_k = 0$, but $\uu \ne \uu'$ such that $\ctrb_\tau[\uu - \uu']_0^\tau = 0$ and $\GGamma_\tau[\uu - \uu']_0^\tau = 0$. Then $[\response_0(\uu - \uu')]_0^\tau = 0$, and $\sigma^\tau(\response_0(\uu - \uu')) = \response(\ctrb_\tau[\uu-\uu']_0^\tau, \sigma^\tau(\uu - \uu')) = 0$, so $\response(\uu) = \response(\uu')$.
    Therefore, $\Sigma$ is not left $\mathcal{U}_\tau(s)$-invertible.
\end{proof}

\begin{proposition}
    $\Sigma$ is left $\mathcal{U}_\tau(s)$-invertible if and only if for any $S \in \Delta_{2s}^\tau$, there exists $z \in \mathbb{C}$ such that $\rank \bm{R}_S(z) = n + \sum_{i \in \range{\tau}} |S_i|$.
\end{proposition}
\begin{proof}
    $(\Rightarrow)$ Suppose that there exists $S \in \Delta_{2s}^\tau$ such that for any $z \in \mathbb{C}$, $\bm{R}_S(z)$ is rank deficient. Denote  the LTI system $\Sigma_S := (\AA^\tau, \ctrb_S, \OO_\tau, \GGamma_S)$, then $\bm{R}_S$ is the Rosenbrock matrix of this system. Hence, by theorem 1, $\Sigma_S$ is not left invertible, so there exists $\vv, \vv': \Natn \to \mathbb{R}^{|S|}$ not equal such that this system's response $\response^{(\Sigma_S)}$ satisfies     $\response_0^{\smash{\smash{{(\Sigma_S)}}}}(\vv - \vv') = 0$. Choose $T, T' \subset \Delta_s^\tau$ such that $T \cup T' = S$, and define $\uu, \uu': \Natn \to \spanof{\Delta_s}$ such that $\uu_k \in \spanof{T_{k \text{ mod } \tau}}$, $\uu'_k \in \spanof{T'_{k \text{ mod } \tau}}$, and $(\uu_k - \uu_k')_{S_{k \text{ mod } \tau}} = \vv_k - \vv_k'$, then $\uu \ne \uu'$ and $\gamma_0(\uu - \uu') = 0$. It follows that $\Sigma$ is not left $\mathcal{U}_\tau(s)$-invertible.

    ($\Leftarrow$) Suppose $\Sigma$ is not left $\mathcal{U}_\tau(s)$-invertible, then there exists $\uu, \uu' \in \mathcal{U}_\tau(s)$ distinct such that $\gamma_0(\uu - \uu') = 0$. Denote $S, S' \in \Delta_s^\tau$ such that $[\uu]_{\tau k}^\tau \in \spanof{S}, [\uu']_{\tau k}^\tau \in \spanof{S'}$, it follows that there exists $\bm{w}, \bm{w}': \Natn \to \mathbb{R}^{|S \cup S'|}$ not equal such that, denoting $\response^{(S \cup S')}$ the response of the system $\Sigma_{S \cup S'} := (\AA^\tau, \ctrb_{S \cup S'}, \OO_\tau, \GGamma_{S \cup S'})$, $\response^{(S \cup S')}_0(\bm{w}) = \response^{(S \cup S')}_0(\bm{w}')$, so $\Sigma_{S \cup S'}$ is not invertible. It follows that for all $z \in \mathbb{C}$, $\bm{R}_{S \cup S'}(z)$ is rank deficient.
\end{proof}

In particular, we obtain a necessary and sufficient characterization of invertibility with respect to inputs with constant support:
\begin{corollary}
    $\Sigma$ is left $\mathcal{U}_1(s)$-invertible if and only if $\forall S \in \Delta_{2s}$, \begin{equation}
        \rank \begin{bmatrix}
            \AA - z\bm{I} & \BB_S \\ \CC & \DD_S
        \end{bmatrix} = n + |S|.
    \end{equation}
\end{corollary}
\vspace{1em}
It is probably clear that, for a system $\Sigma$ to be left $s$-sparse invertible, it must be left $\mathcal{U}_\tau(s)$ invertible for all $\tau$. However, there is no guarantee that generic $s$-piecewise sparse inputs will have periodic supports. Our final result shows that we may bound the required $\tau$ to check, by considering the size of the strongly reachable subspace arrangement $c(\mathcal{T}(2s))$.

\begin{proposition}\label{spectralcond}
    Suppose $\mathcal{T}(2s) = \bigcup_{i \in I} V_i$, and let $\PPsi_i$ be a basis for $V_i$. Then $\Sigma$ is left $s$-sparse invertible if and only if $\forall i \in I, \forall \tau \le c(\mathcal{T}(2s))$, $\forall S \in \Delta_{2s}^\tau, \forall z \in \mathbb{C}$, 
    \begin{equation}
        \rank \begin{bmatrix}
            (\AA^\tau - z\bm{I})\PPsi_i & \ctrb_S \\ \OO_\tau\PPsi_i & \GGamma_S
        \end{bmatrix} = \dim V_i + \sum_{k \in \range{\tau}} |S_k|.
    \end{equation}
\end{proposition}

\begin{proof}
    $(\Rightarrow)$ Suppose there exists $i \in I$, $\tau$, $S \in \Delta_s^\tau$, and $z \in \mathbb{C}$ such that the rank condition fails. Then there exists $\xx \in \mathcal{T}(2s)$ and $U, V \in \spanof{\Delta_s}^\tau$ such that $\AA^\tau\xx + \ctrb_\tau(U - V) = z\xx$ and $\OO_\tau\xx + \GGamma_\tau(U - V) = 0$. Let $\uu, \vv: \Natn \to \spanof{\Delta_s}$ be defined such that for some $M$, $\xx = \ctrb_M[\uu - \vv]_0^M, \GGamma_M[\uu - \vv]_0^M = 0$, this is possible as $\xx \in \mathcal{T}(2s)$. Further, $\forall k \in \Natn$ we have that $[\sigma^M(\uu)]_{k\tau}^\tau = z^k U$ and $[\sigma^M(\vv)]_{k\tau}^\tau = z^k V$, it follows that $\response_0(\uu) = \response_0(\vv)$, hence the system is not invertible.

    $(\Leftarrow)$ Suppose that $\Sigma$ is not left $s$-sparse invertible. If there exists $\tau$ and $S \in \Delta_{2s}^\tau$ such that $\begin{bmatrix}
        \ctrb_S \\ \GGamma_S
    \end{bmatrix}$ is rank deficient, then we have the claim. 
    So in light of proposition \ref{propinv}, assume instead that $\mathcal{T}(2s) \cap \mathcal{V}(2s) \ne 0$. Let $\xx_0 \in \mathcal{T}(2s) \cap \mathcal{V}(2s)$, there exists an input $\uu: \Natn \to \spanof{\Delta_{2s}}$ such that $\forall k \in \Natn, \xx_k := \AA^k\xx_0 + \ctrb_k[\uu]_0^k \in \mathcal{T}(2s) \cap \mathcal{V}(2s) \subset \mathcal{T}(2s)$. 
    Since $\mathcal{T}(2s)$ is a finite subspace arrangement, there exists a subspace $V \subset \mathcal{T}(2s)$ which occurs twice in this trajectory within $c(\mathcal{T}(2s))$ time steps.
    Therefore, denoting $\xx_{k_0}, \xx_{k_1} \in V$ the points where this trajectory passes through $V$ and $\tau = k_1 - k_0$, there exists $S \in \Delta_{2s}^\tau$ and $U$ such that $\AA^\tau\xx_{k_0} + \ctrb_S U = \xx_{k_1}$ and $\OO_\tau\xx_{k_0} + \GGamma_SU = 0$. 
    It may then be shown based on the subspace-preserving property of the iteration \eqref{alg2} that there exists a linear map $\bm{F}: V \to \spanof{S}$ such that $(\AA^\tau + \ctrb_\tau\bm{F})V \subset V$ and $(\OO_\tau + \GGamma_\tau\bm{F})V = 0$. 
    As it is therefore an invariant subspace, $V$ contains an eigenvector $\vv$ of $\bm{A}^\tau + \ctrb_\tau\bm{F}$. 
    Denote $\PPsi_i$ a basis for $V$, and $\vv = \PPsi_i\xx$ and $\tilde{U} = \bm{F}\vv \in \spanof{S}$, we have that there exists $z \in \mathbb{C}$ satisfying
    \[
        \AA^\tau\PPsi_i\xx + \ctrb_\tau \tilde{U} = z\PPsi_i\xx,\quad  \OO_\tau\PPsi_i\xx + \GGamma_\tau \tilde{U} = 0
    \]
    We then conclude $\begin{bmatrix}
        (\AA^\tau - z\bm{I})\PPsi_i & \ctrb_S \\ \OO_\tau\PPsi_i & \GGamma_S
    \end{bmatrix}$ is not full rank.
\end{proof}

This result may be alternatively characterized without detailed knowledge of $\mathcal{T}(2s)$, using only its size and the strong $2s$-sparse reachability index:

\begin{corollary}
    $\Sigma$ is left $s$-sparse invertible if and only if $\forall N \le \tau_{2s} + c(\mathcal{T}(2s)), \forall M < N$, $\forall S \in \Delta_{2s}^N$, $\forall z \in \mathbb{C}$, 
    \begin{equation}
        \rank \begin{bmatrix}
            \GGamma_S \\ \ctrb_S - z\begin{bmatrix}
                \ctrb_{[S]_0^M} & 0
            \end{bmatrix}
        \end{bmatrix} = \sum_{i \in \range{N}} |S_i|.
    \end{equation}
\end{corollary}

\section{Example: Network with Edge Attacks}

\begin{figure}
    \centering
    \vspace{1em}\includegraphics[width=0.8\linewidth]{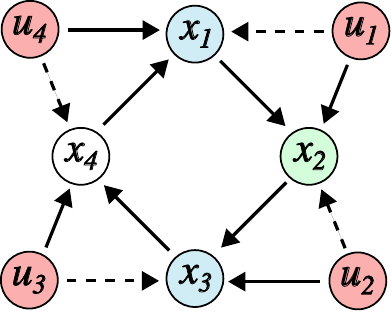}
    \caption{The system $\Sigma_\alpha$ depicted as a network. Solid black arrows indicate a weight of $+1$, dashed indicates $-1$. Red vertices are inputs, and others are states. Blue indicates the vertex state is included in the measurement $\CC_\alpha\xx$, and green indicates inclusion when $\alpha = 1$.}
    \label{fig:enter-label}
\end{figure}

In this section, we demonstrate our results on two linear systems, one which is left $1$-sparse invertible and one which is not, illustrating the three primary characterizations of left $s$-sparse invertibility introduced in section \ref{section:inv}, as well as an instance of a nontrivial weakly unobservable subspace arrangement $\mathcal{V}(1)$.

Consider the following system $\Sigma_\alpha = (\AA, \BB, \CC_\alpha, \DD)$, where $\alpha \in \{0,1\}$, depicted as a network in figure \ref{fig:enter-label}:
\[
    \left[\begin{array}{c|c}
        \AA & \BB \\
        \hline
         \CC_\alpha & \DD
    \end{array}\right] = \left[ \begin{array}{cccc|cccc}
         0 & 0 & 0 & 1 & -1 & 0 & 0 & 1  \\
         1 & 0 & 0 & 0 & 1 & -1 & 0 & 0 \\
         0 & 1 & 0 & 0 & 0 & 1 & -1 & 0 \\
         0 & 0 & 1 & 0 & 0 & 0 & 1 & -1 \\
         \hline
         1 & 0 & 0 & 0 & 0 & 0 & 0 & 0 \\
         0 & \alpha & 0 & 0 & 0 & 0 & 0 & 0 \\
         0 & 0 & 1 & 0 & 0 & 0 & 0 & 0 
    \end{array}\right].
\]
The dynamics simply permute the state of four nodes, and $\CC_\alpha$ measures the state of nodes $1$ and $3$, and $2$ if $\alpha=1$. The inputs may be thought of as edge attacks which could ``interfere'' with the transmission of information from a given node to the next, by changing the apparent amount of information transmitted.
Since $3 = p < m = 4$ in this example, the system is not classically left invertible, but here we will suppose that inputs are $1$-sparse. We will demonstrate that the system is not left $1$-sparse invertible with $\alpha=0$ but is left $1$-sparse invertible when $\alpha=1$, using propositions \ref{propinv}, \ref{rankprop}, and \ref{spectralcond}.

\subsection{$\alpha = 0$, Geometric Characterization}
Consider the state $\xx_k = \bm{e}_4 \in \ker \bm{C}_0$. By choosing $\uu_k = -\bm{e}_4$, $\xx_{k+1} = \AA\xx_k + \BB\uu_k = \bm{e}_1 + \bm{e}_4 - \bm{e}_1 = \bm{e}_4$. Hence, for all $s \ge 1$, $\spanof{4} \subset \mathcal{V}(s)$. Likewise, $\xx_k = \bm{e}_2$ implies $\xx_k \in \ker \bm{C}_0$, and a choice of $\uu_k = -\bm{e}_2$ results in $\xx_{k+1} \in \bm{C}_0$. So $\spanof{2} \subset \mathcal{V}(s)$. However, suppose $\xx_k \in \spanof{2, 4} \setminus (\spanof{2} \cup \spanof{4})$, then no 1-sparse input can result in $\xx_{k+1} \in \ker \CC_0$, but there does exist a 2-sparse input satisfying this requirement. It follows that $\mathcal{V}(1) = \spanof{2} \cup \spanof{4}$ and $\mathcal{V}(2) = \spanof{2, 4}$.

Observe that $\ker \GGamma_1 = \ker\DD = \mathbb{R}^m$, so $\BB\spanof{\Delta_2} \subset \mathcal{T}(2)$. In particular, it contains $\BB(\bm{e}_1 + \bm{e}_4) = \bm{e}_2 - \bm{e}_4$, which is also contained in $\mathcal{V}(2)$. Therefore, by proposition \ref{propinv}, $\Sigma_0$ is not $1$-sparse invertible.

\subsection{$\alpha = 1$, Rank-Based Characterization}
Consider that the matrix
\begin{align}
    \CC_1\BB = \begin{bmatrix}
        -1 & 0 & 0 & 1\\
        1 & -1 & 0 & 0 \\
        0 & 1 & -1 & 0
    \end{bmatrix}
\end{align}
satisfies $\rank \CC_1\BB_{S_0} = |S_0|$ for any $S_0 \in \Delta_{2}$. Since $\DD = 0$, for any other $S_1 \in \Delta_2$,  $\rank \GGamma_{(S_0, S_1)} - \GGamma_{S_1} = \rank \CC_1\BB_{S_0} = |S_0|$, so by proposition \ref{rankprop}, $\Sigma_1$ is left $1$-sparse invertible with delay 1.

\subsection{$\alpha = 0$, Spectral Characterization}
Consider the strongly $2s$-sparse reachable point $\xx_0 = \bm{e}_2 - \bm{e}_4$. The input $\uu_0 = \bm{e}_3 - \bm{e}_1$ results in $\xx_1 = \AA\xx_0 + \BB\uu_0 = \bm{e}_4 - \bm{e}_2$. Let $\PPsi_i$ denote a basis for the subspace of $\mathcal{T}(2s)$ to which $\xx_0$ belongs, we may write $\xx_0 = \PPsi_i \vv_0$. Setting $z = -1$, we have that 
\begin{equation}
\begin{bmatrix}
    \AA\PPsi_i & \BB \\ \CC\PPsi_i & \DD
\end{bmatrix}\begin{bmatrix}
    \vv_0 \\ \uu_0
\end{bmatrix} = \begin{bmatrix}
    z \PPsi_i\vv_0 \\ 0
\end{bmatrix}
.
\end{equation}
By proposition \ref{spectralcond}, we may conclude that $\Sigma_0$ is not left $1$-sparse invertible.

\section{Discussion}
As a collection of necessary and sufficient conditions for the well-posedness of sparse recovery problems, this work deals with conditions that are by their nature computationally hard \cite{mccormick1983combinatorial}. However, they provide a natural scaffolding to deal with sparse inversion problems, and in particular for structured and networked systems, where the inherent delay of a system can be generically nontrivial \cite{garin2022generic}. We expect that, as in the case of the classical $\ell_1$ relaxation for sparse recovery, optimal and computationally tractable implementations will arise as relaxations of appropriate combinatorial problems in this setting. The following are interesting future directions and applications of our results:

\subsection{Bounding constants associated with $\mathcal{T}(s), \mathcal{V}(s)$}
The difficulty of verifying the conditions in this work is determined by the size of $\mathcal{V}(s), \mathcal{T}(s)$ and their indices $\nu_s, \tau_s$. We expect that the magnitude of $\nu_s$ implied by the proof of proposition \ref{prop:finite} is a dramatic overshoot in most cases, and future work should seek to establish general bounds.

\subsection{Connections to Inversion of Switched Systems}
One may formally identify linear systems with sparse inputs with a class of switched systems: consider the support $S: \Natn \to \Delta_{s}$ to be an unknown switching signal determining the time varying system $\Sigma_S(k) = (\AA, \BB_{S_k}, \CC, \DD_{S_k})$. We expect that our results on invertibility may therefore be generalized to a class of switched systems, where switching is restricted to the $\BB$ and $\DD$ matrices.

\subsection{Generalization to Inputs taking values in Subspace Arrangements}
All of our results use properties of sparsity which are also features of subspace arrangements generally. We therefore expect immediate generalization to the setting where $\uu: \Natn \to \mathcal{U} \subset \mathbb{R}^m$, and $\mathcal{U} = \bigcup_{i \in I}U_i$ is a finite subspace arrangement. In the case that $\mathcal{U}$ is a subspace arrangement that is not finite--for example, $\mathcal{U}$ taken to be the set of rank 1 matrices--generalization is less clear, and is an interesting direction for future work. Such results could 

\subsection{Strong Observability and Unknown Input Observers}
Strong observability, the ability to recover the initial state in finite time in the presence of unknown inputs, may be geometrically characterized for linear systems as having a trivial weakly unobservable subspace. Analogously, $\mathcal{V}(2s) = 0$ is necessary and sufficient for initial state recovery in the presence of unknown $s$-sparse inputs. Exploring a characterization of strong detectability, known to be a necessary and sufficient condition for the existence of an unknown input observer for linear systems, is a direction of interest in the sparse case.

\subsection{Right Invertibility}
Right invertibility is formally the ability to construct an input and initial condition which produces any desired output. Due to space constraints, we have restricted our exposition in this work to focus on left invertibility; however, we expect similar arguments to lead to characterizations of right invertibility in the sparse input setting.

\section{Conclusions}
In this work, we have established the first necessary and sufficient conditions for the left invertibility of linear systems with sparse inputs. Leveraging properties of the novel weakly unobservable and strongly reachable subspace arrangements, these characterizations echo the fundamental characterizations of invertibility for standard linear systems. We expect that these characterizations will lead to the generalization of a variety of techniques for inversion of linear systems, and enable a systematic approach to related problems in sparse control.

\bibliography{IEEEabrv,references.bib}

\end{document}